\documentclass[11pt]{article}
\usepackage{etex}
\usepackage{amsmath,amssymb,amsfonts,amsthm}
\usepackage{a4wide}
\usepackage{color}
\usepackage{verbatim}
\usepackage{epsfig,subfigure,epstopdf,float}
\usepackage[]{graphics}
\usepackage{graphicx,inputenc}
\usepackage{subfigure}
\usepackage[justification=centering]{caption}
\usepackage{pictex}
\usepackage{wrapfig}
\usepackage{multirow}
\usepackage[T1]{fontenc}
\usepackage{authblk}
\usepackage[textwidth=6.0in,textheight=9in]{geometry}

\usepackage[colorlinks,linktocpage,linkcolor=blue]{hyperref}
\usepackage{fancyhdr}

\newtheoremstyle{exampstyle}
  {\topsep} 
  {\topsep} 
  {\itshape} 
  {} 
  {\bfseries} 
  {.} 
  {.5em} 
  {} 

\theoremstyle{exampstyle}
\numberwithin{equation}{section}
\newtheorem{theorem}{Theorem}
\newtheorem{lemma}{Lemma}[section]
\newtheorem{assumption}[lemma]{Assumption}

\newtheorem{corollary}[lemma]{Corollary}

\allowdisplaybreaks
\usepackage{parskip}
\let\oldref\ref
\renewcommand{\ref}[1]{(\oldref{#1})}  
\renewcommand{\eqref}[1]{(\oldref{#1})}

\newbox\boxaddrone \newbox\boxaddrtwo

\newcommand{\norm}[1]{\left\Vert#1\right\Vert}

\newcommand{\N}{\mathcal{N}}

\newcommand{\ba}{\begin{eqnarray*}}
\newcommand{\ea}{\end{eqnarray*}}
\newcommand{\be}{\begin{equation}}
\newcommand{\ee}{\end{equation}}
\newcommand{\bea}{\begin{eqnarray}}
\newcommand{\eea}{\end{eqnarray}}

\newbox\boxaddrone \newbox\boxaddrtwo

\def\N+{n\in\mathbb{N}^{+}}

\def\1d{\mathcal{D}((-\Delta)^{\gamma_1+1/2})}
\def\2d{\mathcal{D}((-\Delta)^{\gamma_2+1})}

\def\la{\langle}
\def\ro{\rangle_{L^2(\Omega)}}
\def\rp{\rangle_{L^2(\partial\Omega)}}

\def\re{\operatorname{Re}}
\def\s{\text{Span}}

\begin{document}
\title{Uniqueness and numerical inversion in the time-domain fluorescence diffuse optical tomography}

\author[1,2]{Chunlong Sun\thanks{sunchunlong@nuaa.edu.cn}}
\author[3]{Zhidong Zhang\thanks{zhangzhidong@mail.sysu.edu.cn}}
\affil[1]{College of Science, Nanjing University of Aeronautics and Astronautics, Nanjing 211106, Jiangsu, China}
\affil[2]{Nanjing Center for Applied Mathematics, Nanjing 211135, Jiangsu, China}
\affil[3]{School of Mathematics (Zhuhai), Sun Yat-sen University, Zhuhai 519082, Guangdong, China}

\maketitle

\begin{abstract}
\noindent This work considers the time-domain fluorescence diffuse optical tomography (FDOT). We recover the distribution of fluorophores in biological tissue by the boundary measurements. With the Laplace transform and the knowledge of complex analysis, we build the uniqueness theorem of this inverse problem. After that, the numerical reconstructions are considered. We introduce a non-iterative inversion strategy by peak detection and an iterative inversion algorithm under the framework of regularizing scheme, then give several numerical examples in three-dimensional space illustrating the performance of the proposed inversion schemes.\\


\end{abstract}

\section{Introduction.}
\setcounter{equation}{0}
\subsection{Background and literature.} 
 Recently, fluorescence diffuse optical tomography (FDOT) is rapidly gaining acceptance as an important diagnostic and monitoring tool of symptoms in medical applications \cite{NitziachristoV02, Mycek03, Rudin13}.  It aims to recover the fluorophores quantitatively from some  measurements specified on the medium surface. Compared to X-ray computed tomography (CT), magnetic resonance imaging (MRI) and positron emission tomography (PET), FDOT possesses some advantages like low cost and portable, and therefore it has drawn more and more attention.

For the FDOT, two physical processes are coupled, namely, excitation and fluorescence (emission), which are described by the photon transport model. The system of Maxwell's equations is the rigorous photon transport model to describe the light propagation. However, since the wave property of the photon is lost by multiple scattering in strongly turbid medium such as biological tissue, the photon propagation through turbid media for both excitation and emission can be governed by the Boltzmann radiative transfer equation or shortly the radiative transfer equation (RTE) \cite{Durduran10,Jiang11,Marttelli}. This equation describes the propagation of photon in scattering media and is given as:  
\begin{equation}\label{rt}
\begin{aligned}
c^{-1}\partial_t\Phi(x,\theta,t)=&-\theta \cdot \nabla\Phi(x,\theta,t)-\big(\mu_a+\mu_f+\mu_s\big)\Phi(x,\theta,t) \\
&+ \mu_s\int_{\mathbb{S}^2}\eta(\theta,\theta')\Phi(x,\theta',t)d\theta'+S(x,\theta,t),
\end{aligned}
\end{equation}
where $\Phi(x,\theta,t)$ is the {\it radiant intensity} at $x$ in time $t$, and $\theta$ is a unit vector pointing in the direction of interest. The absorption and scattering coefficients of medium, denoted by $\mu_a(x)$ and $\mu_s(x)$, are the inverses of the mean free paths for absorption and scattering, respectively. Here $\mu_f(x)$ denotes the absorption coefficient of the fluorophores, and $c$ is the speed of photon inside the medium. The scattering phase function $\eta(\theta,\theta')$, which describes the probability that a photon traveling in direction $\theta$ is scattered within the unit solid angle around the direction $\theta'$, is given by 3-dimensional Henyey-Greenstein function \cite{Marttelli}.

However, the cost of solving RTE \eqref{rt} is extremely high, and it is shown that the diffusion equation (DE) can be a  sufficiently accurate approximation to the RTE but has much lower computational cost. Therefore people usually replace the RTE by the diffusion equation. The basic idea of the diffusion approximation is: when scattering is much stronger than absorption, the radiant intensity $\Phi$ can be expressed as an isotropic {\it photon density} plus a small {\it photon flux}, and sequentially the transport equation \eqref{rt} can be reduced to a diffusion equation (see Section \ref{sec_pre_model}). We consider the diffusion approximation for both excitation and emission, and give the propagation of excitation light and emission (fluorescence) light as 
\begin{equation}\label{excitation-Ue}
\begin{cases}
\begin{aligned}
\left(c^{-1}\partial_t -\mu_D\Delta +\mu_a \right)u_e&=0, && (x,t)\in\Omega\times (0,\infty),\\
u_e(x,0)&=0, && x\in\Omega,\\
\frac{\partial u_e}{\partial\nu}+\beta u_e&=\delta(x-x_s)\delta(t), && (x, x_s,t)\in\partial\Omega\times\partial\Omega\times(0,\infty),
\end{aligned}
\end{cases}
\end{equation}
and
\begin{equation}\label{emission-um}
\begin{cases}
\begin{aligned}
\left(c^{-1}\partial_t -\mu_D\Delta +\mu_a \right) u_m&=S_0[\mu_f,u_e](x, t; x_s), && (x,t)\in\Omega\times (0,\infty),\\
u_m(x,0)&=0, && x\in\Omega,\\
\frac{\partial u_m}{\partial\nu}+\beta u_m&=0, && (x,t)\in\partial\Omega\times(0,\infty),
\end{aligned}
\end{cases}
\end{equation}
where $u_e(x,t)$ and $u_m(x,t)$ are the photon densities of excitation light and emission light, respectively. Here $\Omega\subset\mathbb R^d$ is an open bounded subset with sufficiently smooth boundary; the vanishing initial condition and Robin impedance boundary condition with $\beta>0$ are used; $\mu_D$ is the photon diffusion coefficient defined by $\mu_D:=1/(3\mu_s(1-g))$ where $g$ is the anisotropy parameter. We set $\mu_D$ and $\mu_a$ as positive constants in this work. The source term $S_0$ for $u_m$ on the right-hand side of \eqref{emission-um} contains the excitation $u_e$ and is specified by
\begin{equation}\label{source} 
S_0[\mu_f,u_e](x,t;x_s)=\frac{\mu_f(x)}{\tau}\int_0^t e^{-\frac{t-s}{\tau}} u_e(x,s;x_s) \, ds,
\end{equation} 
where $\tau > 0$ is the fluorescence lifetime.

The FDOT based on RTE model or DE model has drawn extensive attention of researchers in recent years. For such inverse problems, three kinds of measurements, like time domain \cite{Gao08, Lam05}, continuous wave \cite{Ducros11, Liu20,Patwardhan05} and frequency domain \cite{Corlu-etal07,Nitziachristos01} are used. For these optical imaging problems, the necessary regularization techniques such as Tikhonov regularization, sparse regularization methods and  hybrid regularization methods have been introduced to overcome the ill-posedness of the inverse problems \cite{Correia10,Dutta12,Liu20}. Especially, the sparse regularization methods have additional advantages in promoting sparsity and higher spatial resolution for the cases that the target is relatively small compared to the background. Further, aiming to the best reconstruction result, one must consider how to place the source-detector pairs onto the object surface. To answer this problem, the definition of admissible set of possible sensor patterns on the boundary and a proper optimal condition are usually needed \cite{Bergounioux19,Demidenko05,Dutta10,Hyvonen14}. 

Above all, most existing works focus on improving the quality of the reconstruction, especially the spatial resolution and the reconstruction speed. However, in the aspect of theorical analysis, the uniqueness of FDOT has not been rigorously investigated as far as we know. In this work, we will study the uniqueness of time-domain FDOT. The details are contained in the main theorem, which is stated in the next subsection.

\subsection{The main theorem.} 
In this work, we consider the time-dependent fluorophores which means that the absorption coefficient $\mu_f$ depends on both $x$ and $t$. We suppose that $S_0[\mu_f,u_e](x,t;x_s)$ in  \eqref{source} possesses the following semi-discrete formulation: 
\begin{equation}\label{ua-time}
S_0[\mu_f,u_e]=\sum_{k=1}^K p_k(x)\chi_{{}_{t\in [t_{k-1},t_k)}},
\end{equation} 
where the time mesh $\{t_k\}_{k=0}^K$ is given and the spatial components $\{p_k(x)\}_{k=1}^K$ are undetermined. For the mesh $\{t_k\}_{k=0}^K$, $K$ can be infinity and $\inf\{|t_k-t_{k+1}|:k=0,\cdots,K-1\}>0$; for the unknown $\{p_k(x)\}_{k=1}^K$, we consider to recover them in $L^2(\Omega)$, i.e. set $\{p_k(x)\}_{k=1}^K\subset L^2(\Omega)$. The representation \eqref{ua-time} contains the information of the time evolution process of the unknown fluorophores. Since the excitation $u_e$ is known, we can extract the information of $\mu_f$ from the components $\{p_k(x)\}_{k=1}^K$. The measurements we used are the boundary data given by  
\begin{equation}\label{Um-measure}
u_m(x_d,t;x_s),\quad x_d\in \Gamma_d\subset\partial\Omega, \ x_s\in \Gamma_s\subset\partial\Omega, \ t\in(0,\infty),
\end{equation}
where $\Gamma_s$ and $\Gamma_d$ are the excitation area and observed area, respectively. From the Robin boundary condition, we can get the flux $\frac{\partial u_m}{\partial\nu}\big|_{\Gamma_d\times(0,\infty)\times\Gamma_s}$ from the measurements \eqref{Um-measure}. 
Hence, in this work, the so-called time-domain FDOT based on the initial-boundary value problems \eqref{excitation-Ue}-\eqref{emission-um} is to solve the inverse problem:
\begin{equation*}
\text{recovering} \ \{p_k(x)\}_{k=1}^K \ \text{in \eqref{ua-time} from the measurements \eqref{Um-measure} or the flux}\  \frac{\partial u_m}{\partial\nu}\big|_{\Gamma_d\times(0,\infty)\times\Gamma_s}.
\end{equation*}

Moreover, if either of $\mu_D$ and $\mu_a$ is also unknown and we want to recover the unknown coefficient and the source $\mu_f$ simultaneously, an extra condition needs to be set. Letting $(\lambda_1,\varphi(x))$ be one principal eigenpair of the operator $-\Delta$ on $\Omega$ with Robin boundary condition, we give the following condition:   
\begin{equation}\label{condition_mu_a}
 \la p_1(\cdot),\varphi(\cdot)\ro\ne 0.
\end{equation}
Here the notation $\la\cdot,\cdot\ro$ means the inner product in space $L^2(\Omega)$. For the details of the eigensystem of $-\Delta$ on $\Omega$, see Section \ref{subsec_eigen}. Also we give Assumption \ref{condition_p_regularity} in Section \ref{subsec_eigen}, which is referred in the statement of the uniqueness theorem. Furthermore, to make our analysis more convenient, we set $c=1$ in models  \eqref{excitation-Ue}-\eqref{emission-um}. Now it is time to state the uniqueness theorem. 

\begin{theorem}\label{uniqueness}
Set Assumption \ref{condition_p_regularity} be valid, and  $\Gamma_s$ and $\Gamma_d$ be a single point and a nonempty open subset of boundary $\partial\Omega$, respectively.   
For the two sets of unknowns $\big\{\mu_D,\mu_a, \{p_k(x)\}_{k=1}^K \big\}$ and $\big\{\tilde\mu_D,\tilde\mu_a,\{\tilde p_k(x)\}_{k=1}^K \big\}$, we denote the corresponding solutions of equation  \eqref{emission-um} by $u$ and $\tilde u$ respectively, and are provided $$
\frac{\partial u}{\partial\nu}(x_d,t;x_s)=\frac{\partial \tilde u}{\partial\nu}(x_d,t;x_s),\quad (x_d,x_s,t)\in\Gamma_d\times\Gamma_s\times(0,\infty).
$$ 
Then we conclude that:
\begin{itemize}
 \item [(\romannumeral 1)] given $\mu_a=\tilde \mu_a$ and 
 $\mu_D=\tilde \mu_D$, we have 
 $$p_k(x)=\tilde p_k(x)\ \text{in}\ L^2(\Omega)\ \text{for}\ k=1,\cdots, K;$$
 \item [(\romannumeral 2)] with condition \eqref{condition_mu_a} and $\mu_D=\tilde \mu_D$, we have  $$\mu_a=\tilde \mu_a,\ p_k(x)=\tilde p_k(x)\ \text{in}\ L^2(\Omega)\ \text{for}\ k=1,\cdots, K;$$
  \item [(\romannumeral 3)] with condition \eqref{condition_mu_a} and $\mu_a=\tilde \mu_a$, we have  $$\mu_D=\tilde\mu_D,\ p_k(x)=\tilde p_k(x)\ \text{in}\ L^2(\Omega)\ \text{for}\ k=1,\cdots, K.$$ 
\end{itemize}
\end{theorem}

From Theorem \ref{uniqueness}, we can derive the following corollary straightforwardly. 
\begin{corollary}\label{uniqueness-x}
Under the conditions of Theorem \ref{uniqueness}, the boundary data \eqref{Um-measure} can identify the stationary fluorophores uniquely.
\end{corollary}
The above corollary considers the case of stationary fluorescence target, i.e. $\mu_f:=\mu_f(x)$. This is the common problem in the community of fluorescence tomography, which is covered by Theorem \ref{uniqueness}.

\subsection{Contribution and outline.}
The literature on the uniqueness of time-domain FDOT is relatively rare and the existing conclusions seem not to be convenient in practical applications. Here we discuss \cite{Liu:2020} as an example. In this article, the authors aim to identify the absorption coefficient $\mu_f(x)$ from time-resolved boundary measurements. But this work  requires strong prior assumptions on $\mu_f(x)$. The authors prove that for $\Omega=\mathbb{R}_+^3$, $x=(\tilde x, x_3)\in \mathbb{R}^2\times\mathbb{R}_+^1$ and small lifetime $\tau$, by supposing $\mu_f(x)$ has the variable separable form $\mu_f(x)=p(\tilde x)q(x_3)$ with known vertical information $q(x_3)$, the horizontal information $p(\tilde x)$ can be uniquely determined from the boundary measurements \eqref{Um-measure} with $(x_d,x_s)\in \{x_d^0\}\times \partial\Omega$ or $(x_d,x_s)\in \partial\Omega\times\{x_s^0\}$. Here $x_d^0$ and $x_s^0$ are the single points on $\partial\Omega$. However, we see that the used measurements require either the source point $x_s$ or the detector point $x_d$ to cover the whole boundary surface, which implies the impractical cost in  applications.

In this work, the uniqueness of time-domain FDOT in a more general framework is established. The concerned unknown $\mu_f$ has a more general formulation, and we use the sparse boundary measurements (we call the used measurements as sparse boundary measurements since the observed area $\Gamma_d$ can be an arbitrarily open subset of boundary $\partial\Omega$ and the excitation area $\Gamma_s$ only needs to be a single point). Theorem \ref{uniqueness} confirms that the semi-discrete unknown $\mu_f$ can be uniquely determined by the sparse boundary measurements. Therefore, this conclusion is of great significance in practical applications. Moreover, Theorem \ref{uniqueness} also concerns the cases of recovering $\mu_f$ and one of $\mu_a$ or $\mu_D$ simultaneously. The reconstructions of $\mu_a$ and $\mu_D$ are known as the diffuse optical tomography (DOT). Some aspects of uniqueness and ill-posedness of DOT are considered \cite{Anikonov84,Anikonov85,Arridge98,Bal09,Choulli,Harrach09,He2000}. Theorem \ref{uniqueness} gives partial uniqueness result for the DOT with fluorescence but it requires constant absorption and scattering. Hence the main theorem in this work is also useful in DOT to some extent. 

The rest of this article is organized as follows. In Section \ref{sec_pre}, we first introduce the semi-discretized model for our time-domain FDOT based on diffusion equations. Then we collect several preliminary works. We prove the uniqueness theorem in Section \ref{sec_uni}. The Laplace transform and some knowledge of complex analysis will be used. The numerical inversions will be considered in Section \ref{sec_num}. The inversion scheme minimizing the regularized cost functional is implemented by an iterative process. We show the validity of the proposed scheme by several numerical examples in three-dimensional space.

\section{The diffusion approximation and some preliminary results.}\label{sec_pre}

\subsection{The diffusion approximation.}\label{sec_pre_model}
We refer to \cite{Arridge99,Arr09} for the details of diffusion approximation (or $P_1$-approximation) and \cite{Liu:2020} for the error estimates of model approximation. However, we give a sketch for the derivation from RTE to DE for the completeness of this article. 

First, one can expand $\Phi(x,\theta,t)$ and $S(x,\theta,t)$ in RTE \eqref{rt} into spherical harmonics as 
\begin{align*}
\Phi(x,\theta,t)&=\sum_{l=0}^{\infty}\sum_{m=-l}^{l} \left(\frac{2l+1}{4\pi}\right)^{1/2} \Phi_{l,m}(x,t) Y_{l,m}(\theta),\\
S(x,\theta,t)&=\sum_{l=0}^{\infty}\sum_{m=-l}^{l} \left(\frac{2l+1}{4\pi}\right)^{1/2} S_{l,m}(x,t) Y_{l,m}(\theta).
\end{align*}
Keeping only the terms $l\leq 1$ in above expansions, we have the approximation
\begin{equation}\label{sun1-1}
\Phi(x,\theta,t) \approx \frac{1}{4\pi}u(x,t)+\frac{3}{4\pi}\theta\cdot {\mathbf{J}}(x,t), \quad \theta\in\mathbb{S}^2,
\end{equation}
where the quantity \begin{equation*}
u(x,t):=\int_{\mathbb{S}^2}\Phi(x,\theta,t) \, d\theta
\end{equation*}
is called {\it photon density}, and the quantity \begin{equation*}
\mathbf{J}(x,t):=\int_{\mathbb{S}^2} \theta\cdot\Phi(x,\theta,t) \, d\theta
\end{equation*}
is called {\it photon flux}. Similarly, we have
\begin{equation}\label{sun1-2}
S(x,\theta,t) \approx \frac{1}{4\pi}S_0(x,t)+\frac{3}{4\pi}\theta\cdot S_1(x,t), \quad \theta\in\mathbb{S}^2,
\end{equation}
where $S_0(x,t):=\int_{\mathbb{S}^2} S(x,\theta,t) \, d\theta$ and $S_1(x,t):=\int_{\mathbb{S}^2} \theta\cdot S(x,\theta,t) \, d\theta$ are the isotropic component and the first angular moment of source term, respectively. 

Next, by owing \eqref{sun1-1} and \eqref{sun1-2} into the RTE \eqref{rt}, and then integrating the RTE and the RTE multiplied by $\theta$ over ${\mathbb{S}^2}$, we have 
\begin{equation}\label{S0}
\begin{aligned}
\left(c^{-1} \partial_t + \mu_a +\mu_f\right) u(x,t)+\nabla\cdot \mathbf{J}(x,t)  &=S_0, \\
\left(c^{-1}\partial_t + \mu_a+\mu_f+\mu_s' \right) \mathbf{J}(x,t) + \frac{1}{3}\nabla u(x,t)& = S_1,
\end{aligned}
\end{equation}
where $\mu_s'$ is the reduced scattering coefficient defined by $\mu_s':=\mu_s(1-g)$ with $g$ as the anisotropy parameter. By $\partial|\mathbf{J}|/\partial t\ll c\mu_s'|\mathbf{J}|$ and the assumption that sources are isotropic, we arrive at the so-called Fick's law 
\begin{equation*}
\mathbf{J}=-\frac{1}{3(\mu_a+\mu_f+\mu_s')}\nabla u.
\end{equation*}
Since $\mu_a,\,\mu_f \ll \mu_s'$ in the strong scattering medium, we can approximate 
\begin{equation}\label{Fick}
\mathbf{J}=-\frac{1}{3\mu_s'}\nabla u.
\end{equation}

Finally, supposing the absorption of the fluorophores is much smaller than the absorption of the medium, from substituting \eqref{Fick} into \eqref{S0} we immediately have the diffusion equation given by
\begin{equation*}
\left(c^{-1}\partial_t -\nabla \cdot \frac{1}{3\mu_s'(x)} \nabla  +\mu_a(x) \right) u=S_0(x,t).
\end{equation*}

We deduce the above diffusion approximation for both excitation and emission, and suppose that the optical parameters for emission are constants and same to those for the excitation. Then the system  \eqref{excitation-Ue}-\eqref{emission-um} can be derived.

\subsection{Some preliminary results.}
\subsubsection{The eigensystem of $-\Delta$ on $\Omega$.}
\label{subsec_eigen}
For the operator $-\Delta$ on $H^2(\Omega)\cap H_0^1(\Omega)$ with Robin boundary condition, we denote the eigensystem by $\{\lambda_n,\varphi_n(x)\}_{n=1}^\infty$. Then the following properties will be valid: 
\begin{itemize}
    \item $0<\lambda_1\le \lambda_2\le\cdots$ and $\lambda_n\to\infty$ as $n\to \infty$;
    \item $\{\varphi_n(x)\}_{n=1}^\infty$ is an orthonormal basis of $L^2(\Omega)$. 
\end{itemize}
Furthermore, if $\varphi_n$ is an eigenfunction of $-\Delta$ corresponding to $\lambda_n$, so is $\overline{\varphi_n}$,  where $\overline{\varphi_n}$ is the complex conjugate of $\varphi_n$. Hence we have that the set $\{\varphi_n(x)\}_{n=1}^\infty$ coincides with $\{\overline{\varphi_n(x)}\}_{n=1}^\infty$. 
The trace theorem yields that   $\{\frac{\partial\varphi_n}{\partial\nu}|_{\partial \Omega}\}_{n=1}^\infty\subset H^{1/2}(\partial \Omega)$.
Also, we denote $\la\cdot,\cdot\rp$ as the inner product in $L^2(\partial\Omega)$. 

The next lemmas concern the vanishing property and the density of $\frac{\partial\varphi_n}{\partial\nu}$ on $\partial\Omega$. 
\begin{lemma}\label{nonempty_open}
If $\Gamma$ is a nonempty open subset of $\partial \Omega$, then for each $n\in\mathbb N^+$, $\frac{\partial \varphi_n}{\partial\nu}$ can not vanish on $\Gamma$. 
\end{lemma}
\begin{proof}
 See \cite[Lemma 2.1]{LinZhangZhang:2021}.
\end{proof}

\begin{lemma}
 The set $\text{Span}\{\frac{\partial\varphi_n}{\partial\nu}|_{\partial \Omega}\}_{n=1}^\infty$ is dense in $L^2(\partial \Omega)$. 
\end{lemma}
\begin{proof}
 Not hard to see that $H^{3/2}(\partial\Omega)$ is dense in  $L^2(\partial \Omega)$ under the norm $\|\cdot\|_{L^2(\partial\Omega)}$. So it is sufficient to show $\tilde\psi\in H^{3/2}(\partial \Omega)$ vanishes almost everywhere on $\partial \Omega$ if $\la \tilde\psi,\frac{\partial\varphi_n}{\partial\nu}\rp=0$ for $n\in\mathbb{N}^+$.
 
We set $\psi$ be the weak solution of the system below:  
\begin{equation*}
\begin{cases}
\begin{aligned}
    -\Delta \psi(x)&=0, &&x\in \Omega,\\
    \frac{\partial \psi}{\partial\nu}+\beta\psi&=\tilde \psi,&&x\in\partial \Omega.
\end{aligned}
\end{cases}
\end{equation*}
We have $\psi\in H^2(\Omega)$ from the regularity $\tilde\psi\in H^{3/2}(\partial \Omega)$, sequentially the Green's identity can be used. For $n\in\mathbb N^+$, we have 
\begin{align*}
\la -\Delta \psi,\varphi_n\ro-\la \psi, -\Delta\varphi_n\ro 
&=\la \psi,\frac{\partial\varphi_n}{\partial\nu}\rp-\la\frac{\partial\psi}{\partial\nu}, \varphi_n\rp\\
&=\la \psi+\beta^{-1}\frac{\partial\psi}{\partial\nu},\frac{\partial\varphi_n}{\partial\nu}\rp\\
&=\beta^{-1}\la \tilde\psi,\frac{\partial\varphi_n}{\partial\nu}\rp.
\end{align*} 
From $\Delta \psi=0$ on $\Omega$ and the fact $\la \tilde\psi,\frac{\partial\varphi_n}{\partial\nu}\rp=0$, we have $$\la \psi,-\Delta\varphi_n\ro =\lambda_n\la \psi,\varphi_n\ro=0.$$ 
So we have proved that for each $n\in\mathbb N^+$, $\la\psi, \varphi_n\ro=0$. Recalling the completeness of $\{\varphi_n\}_{n=1}^\infty$ in $L^2(\Omega)$, it holds that  $\|\psi\|_{L^2(\Omega)}=0$. 
From the definition of weak derivative and Sobolev space, we have $\|\psi\|_{H^2(\Omega)}=0$. By the  continuity of the trace operator, it gives that 
$$\|\psi\|_{L^2(\partial \Omega)}\le C\|\psi\|_{H^2(\Omega)}=0,\quad 
\Big\|\frac{\partial\psi}{\partial\nu}\Big\|_{L^2(\partial\Omega)}
\le C\|\psi\|_{H^2(\Omega)}=0.$$
This means that $\tilde\psi=0$ almost everywhere on $\partial\Omega$ and the proof is complete. 
\end{proof}

\subsubsection{The set $\{\xi_l\}_{l=1}^\infty$ and the coefficients $\{c_{z,n}\}$.}
From the above lemma, we are allowed to construct the orthonormal basis $\{\tilde \xi_l\}_{l=1}^\infty$ in $L^2(\partial \Omega)$. Firstly we set $\tilde\xi_1=\frac{\partial\varphi_1}{\partial\nu}|_{\partial \Omega}/\|\frac{\partial\varphi_1}{\partial\nu}\|_{L^2(\partial \Omega)},$ and assume that the orthonormal set $\{\tilde\xi_j\}_{j=1}^{l-1}$ has been built for $l=2,3,\cdots$. Then we set $n_l\in\mathbb N^+$ be the smallest number such that  $\frac{\partial\varphi_{n_l}}{\partial\nu}|_{\partial \Omega}\notin \s\{\tilde\xi_j\}_{j=1}^{l-1}$, and pick $\tilde\xi_l\in \s\{\frac{\partial\varphi_{n_l}}{\partial\nu}|_{\partial \Omega},\ \tilde\xi_1, \cdots, \tilde\xi_{l-1}\}$ satisfying  
\begin{equation*}
 \la\tilde\xi_l, \tilde\xi_j\rp=0\ \text{for}\ j=1,\cdots, l-1,\ \text{and}\ \|\tilde\xi_l\|_{L^2(\partial \Omega)}=1.
\end{equation*}
The density of $\text{Span}\{\frac{\partial\varphi_n}{\partial\nu}|_{\partial \Omega}\}_{n=1}^\infty$ in $L^2(\partial \Omega)$ yields that $\{\tilde\xi_l\}_{l=1}^\infty$ is an orthonormal basis in $L^2(\partial \Omega)$. Also, we have 
$\tilde \xi_l\in H^{1/2}(\partial \Omega)$ for each $l\in\mathbb N^+$. 

Next, for $l\in \mathbb N^+$, we define $\xi_l\in H^1( \Omega)$ be the weak solution of the system: 
\begin{equation}\label{PDE_xi}
\begin{cases}
\begin{aligned}
  (-\mu_D\Delta+\mu_a)\xi_l(x)&=0, &&x\in \Omega,\\
    \frac{\partial \xi_l}{\partial\nu}+\beta\xi_l&=\tilde \xi_l,&&x\in\partial \Omega.
\end{aligned}
\end{cases}
\end{equation}
Fixing $z\in\partial \Omega$, we define the series 
$\psi_z^N\in H^1(\Omega)$ as 
\begin{equation}\label{psi_z^N}
 \psi_z^N(x)=\sum_{l=1}^N \tilde\xi_l(z)\overline{\xi_l(x)} , \ x\in  \Omega.
\end{equation}
The definition of coefficients $\{c_{z,n}\}$ is given in the next lemma. 

\begin{lemma}\label{c_z,n}
 For each $z\in\partial \Omega$ and $n\in \mathbb N^+$, $\lim_{N\to \infty}\la \psi_z^N,\overline{\varphi_n}\ro$ exists and we denote the limit by $c_{z,n}$.
\end{lemma}
\begin{proof}
 From Green's identities we have 
 \begin{equation*}
 \begin{aligned}
  \la \psi_z^N,\overline{\varphi_n}\ro&=(\mu_D\lambda_n+\mu_a)^{-1}\la\psi_z^N, (-\mu_D\Delta+\mu_a)\overline{\varphi_n}\ro \\
  &=(\mu_D\lambda_n+\mu_a)^{-1}\Big( \la (-\mu_D\Delta+\mu_a)\psi_z^N,\overline{\varphi_n}\ro-\mu_D\la\psi_z^N,\frac{\partial\overline{\varphi_n}}{\partial\nu}\rp\\
  &\qquad\qquad\qquad\qquad\  +\mu_D\la\frac{\partial\psi_z^N}{\partial\nu},\overline{\varphi_n}\rp\Big) \\
  &=-(\mu_D\lambda_n+\mu_a)^{-1}\mu_D\beta^{-1}\la\frac{\partial\psi_z^N}{\partial\nu}+\beta\psi_z^N,\frac{\partial\overline{\varphi_n}}{\partial\nu}\rp\\
  &=-(\mu_D\lambda_n+\mu_a)^{-1}\mu_D\beta^{-1}\sum_{l=1}^N \tilde\xi_l(z)\la\frac{\partial\varphi_n}{\partial\nu}, \tilde\xi_l\rp=:c_{z,n}^N,
  \end{aligned}
 \end{equation*}
where the system \eqref{PDE_xi} and the boundary condition of $\varphi_n$ are used. From the definition of $\{\tilde\xi_l\}_{l=1}^\infty$, we have                 $\la\frac{\partial\varphi_n}{\partial\nu}, \tilde\xi_l\rp=0$ for large $l$. Hence the value of $c_{z,n}^N$ will not change if $N$ is sufficiently large. This gives that $\lim_{N\to \infty}\la \psi_z^N,\overline{\varphi_n}\ro$ exists and the proof is complete. 
\end{proof}

From the above lemma we can see 
$$c_{z,n}=\lim_{N\to \infty}c_{z,n}^N=\lim_{N\to \infty}\la \psi_z^N,\overline{\varphi_n}\ro,$$ 
also we denote $p_{k,n}:=\la p_k(\cdot),\varphi_n(\cdot)\ro$ for $k=1,\cdots, K$ and $n\in\mathbb N^+$. For the coefficients $\{c_{z,n},c_{z,n}^N, p_{k,n}\}$, we give the following conditions, which will be used in the future proof. 
\begin{assumption}\label{condition_p_regularity}
\hfill
\begin{itemize}
\item [(a)] For $k\in\{1,\cdots,K\}$ and a.e. $z\in\partial\Omega$, we can find $C>0$ which is independent of $N$ such that $\sum_{n=1}^\infty |c^N_{z,n}p_{k,n}|<C<\infty$ for $N\in\mathbb N^+$. 
\item [(b)] $\sum_{k=1}^K\sum_{n=1}^\infty |c_{z,n}p_{k,n}|<\infty$ for a.e. $z\in \partial \Omega$.
\end{itemize}
\end{assumption}

\subsubsection{Auxiliary lemmas from \cite{LinZhangZhang:2021}.}
At the end of this section, we collect some auxiliary results from the reference \cite{LinZhangZhang:2021}. 
\begin{lemma}\label{data_formula}
Assuming that $\{p_k(x)\}_{k=1}^K$ possess appropriate regularities, then for a.e. $z\in\partial \Omega$ and a.e.  $t\in(0,\infty)$, it holds that 
\begin{equation*}
 -\int_0^t \frac{\partial u}{\partial\nu}(z, \tau)\ d\tau
 =\int_0^t \sum_{k=1}^K \chi_{{}_{t-\tau\in [t_{k-1},t_k)}} 
 \Big[\sum_{n=1}^\infty c_{z,n}p_{k,n} 
 (1-e^{-(\mu_D\lambda_n+\mu_a)\tau})\Big]\ d\tau.
 \end{equation*}
 \end{lemma}
\begin{proof}
 This is \cite[Corollary 3.2]{LinZhangZhang:2021}. 
\end{proof}

The next two lemmas are included in \cite[Section 3.3]{LinZhangZhang:2021}. 
\begin{lemma}\label{lemma_uniqueness_1}
We denote the set of distinct eigenvalues with increasing order by $\{\lambda_j\}_{j=1}^\infty$.  
For any nonempty open subset $\Gamma\subset\partial\Omega$, if $$\sum_{\lambda_n=\lambda_j} c_{z,n}\eta_n=0\ \text{for}\  j\in\mathbb{N}^+\ \text{and\ a.e.}\  z\in\Gamma,$$ 
then $\{\eta_n\}_{n=1}^\infty=\{0\}$. 
\end{lemma}

\begin{lemma}
\label{lemma_uniqueness} 
Let $\Gamma$ be a nonempty open subset of $\partial\Omega$ and $\sum_{n=1}^\infty c_{z,n}\eta_n$ be absolutely convergent for a.e. $z\in \Gamma$. Given $\epsilon>0$, then  
\begin{equation*}
\label{eq-exp_eigen}
\lim_{\re s\to\infty}e^{\epsilon s} \sum_{n=1}^\infty c_{z,n}\eta_n(\mu_D\lambda_n+\mu_a)(s+\mu_D\lambda_n+\mu_a)^{-1}=0\ \text{for a.e.}\ z\in\Gamma 
\end{equation*}
leads to $\{\eta_n\}_{n=1}^\infty=\{0\}$. 
\end{lemma}

\section{The proof of the uniqueness theorem.}\label{sec_uni} 
\subsection{The Laplace transform analysis.}
The convolution structure in the result of Lemma \ref{data_formula} 
encourages us to apply the Laplace transform, which is defined as 
\begin{equation*}
 \mathcal L\{\psi(t)\}(s) =\int_0^\infty e^{-st} \psi(t)\ dt,
 \quad s\in\mathbb C.
\end{equation*} 
Not hard to see that for $\re s>0$,  
\begin{equation*}
 \mathcal L \{1-e^{-(\mu_D\lambda_n+\mu_a)t}\}(s)=(\mu_D\lambda_n+\mu_a)s^{-1}(s+\mu_D\lambda_n+\mu_a)^{-1}.
\end{equation*}
From Assumption \ref{condition_p_regularity}, it holds that 
\begin{equation*}
\begin{aligned}
&\Big|\int_0^t \sum_{k=1}^K \chi_{{}_{t-\tau\in [t_{k-1},t_k)}} 
 \Big[\sum_{n=1}^\infty c_{z,n}p_{k,n} 
 (1-e^{-(\mu_D\lambda_n+\mu_a)\tau})\Big]\ d\tau\Big|\\
&\le 2\int_0^t \sum_{k=1}^{K} \big|\chi_{{}_{t-\tau\in[t_{k-1},t_k)}}\big|\ \sum_{n=1}^\infty \big|c_{z,n}p_{k,n}\big|\ d\tau\le  Ct.
 \end{aligned}
\end{equation*}
Also we can see $|e^{-st}t|$ is integrable on $(0,\infty)$ if $\re s>0$. Then by the Dominated Convergence Theorem, taking Laplace transform on the result in Lemma \ref{data_formula} yields that for each $s\in\mathbb C^+:=\{s\in\mathbb C:\re s> 0\},$
\begin{align*}
&\mathcal L \left\{-\int_0^t \frac{\partial u}{\partial\nu}(z,\tau)\ d\tau \right\}(s)\\
&=\sum_{k=1}^{K}\sum_{n=1}^\infty\int_0^\infty e^{-st} \int_0^t  \chi_{{}_{t-\tau\in[t_{k-1},t_k)}} c_{z,n}p_{k,n}(1-e^{-(\mu_D\lambda_n+\mu_a)\tau})\ d\tau\ dt\\
&=s^{-2}\sum_{k=1}^{K}(e^{-t_{k-1}s}-e^{-t_ks})\Big[\sum_{n=1}^\infty c_{z,n}p_{k,n}(\mu_D\lambda_n+\mu_a)(s+\mu_D\lambda_n+\mu_a)^{-1}\Big],
\end{align*}
which with $\mathcal L (\int\psi)=s^{-1}\mathcal L(\psi)$ implies that for $s\in\mathbb C^+$, 
\begin{equation}\label{laplace}
s\mathcal L \left\{ -\frac{\partial u}{\partial\nu}(z,t) \right\}(s)
=\sum_{k=1}^{K}(e^{-t_{k-1}s}-e^{-t_ks})\Big[\sum_{n=1}^\infty c_{z,n}p_{k,n}(\mu_D\lambda_n+\mu_a)(s+\mu_D\lambda_n+\mu_a)^{-1}\Big].
\end{equation}
After deducing \eqref{laplace}, we need to show the well-definedness and analyticity of the complex series in it. In the next lemma, we recall that the distinct eigenvalues with increasing order are denoted by $\{\lambda_j\}_{j=1}^\infty$. 

\begin{lemma}\label{analytic}
 Under Assumption \ref{condition_p_regularity}, the series  $$\sum_{k=1}^{K}(e^{-t_{k-1}s}-e^{-t_ks})\Big[\sum_{n=1}^\infty c_{z,n}p_{k,n}(\mu_D\lambda_n+\mu_a)(s+\mu_D\lambda_n+\mu_a)^{-1}\Big]$$ is analytic on $\mathbb C^+$. 
\end{lemma}
\begin{proof}
Firstly, let us show the analyticity of the series $\sum_{n=1}^\infty c_{z,n}p_{k,n}(\mu_D\lambda_n+\mu_a)(s+\mu_D\lambda_n+\mu_a)^{-1}$ on $\mathbb C\setminus \{-\mu_D\lambda_j-\mu_a\}_{j=1}^\infty$ for $k=1,\cdots,K$. Obviously we see that $c_{z,n}p_{k,n}(\mu_D\lambda_n+\mu_a)(s+\mu_D\lambda_n+\mu_a)^{-1}$ is holomorphic on $\mathbb C\setminus \{-\mu_D\lambda_j-\mu_a\}_{j=1}^\infty$. So it is sufficient to show the uniform convergence of the above series. 

For $R>0$, we define $\mathbb C_R:=\{s\in \mathbb C:|s|< R\}$. Recalling that $\lambda_n\to \infty$, then there exists a large $N_1>0$ such that $\mu_D\lambda_n+\mu_a>2R$ for $n\ge N_1$. Sequentially, for $s\in \mathbb C_R\setminus \{-\mu_D\lambda_j-\mu_a\}_{j=1}^\infty$ and $n\ge N_1$, we have 
\begin{equation*}
 |s+\mu_D\lambda_n+\mu_a|\ge |\re s+\mu_D\lambda_n+\mu_a|=\re s+\mu_D\lambda_n+\mu_a \ge \mu_D\lambda_n+\mu_a -R,
\end{equation*}
which gives 
\begin{equation*}
 |(\mu_D\lambda_n+\mu_a)(s+\mu_D\lambda_n+\mu_a)^{-1}|\le(\mu_D\lambda_n+\mu_a)(\mu_D\lambda_n+\mu_a -R)^{-1}< 2.
\end{equation*}
Given $\epsilon>0,$ Assumption \ref{condition_p_regularity} yields that there exists $N_2>0$ such that 
for $l\ge N_2$, $\sum_{n=l}^\infty|c_{z,n}p_{k,n}|<\epsilon.$ So, for $l\ge \max\{N_1,N_2\}$ and $s\in \mathbb C_R\setminus \{-\mu_D\lambda_j-\mu_a\}_{j=1}^\infty$, 
\begin{equation*}
 \Big|\sum_{n=l}^\infty c_{z,n}p_{k,n}(\mu_D\lambda_n+\mu_a)(s+\mu_D\lambda_n+\mu_a)^{-1}\Big| \le 2\sum_{n=l}^\infty |c_{z,n}p_{k,n}|<2\epsilon,
\end{equation*}
which implies the uniform convergence. With this uniform convergence result, we see that the series
$\sum_{n=1}^\infty c_{z,n}p_{k,n}(\mu_D\lambda_n+\mu_a)(s+\mu_D\lambda_n+\mu_a)^{-1}$ is holomorphic on  $\mathbb C_R\setminus \{-\mu_D\lambda_j-\mu_a\}_{j=1}^\infty$ for each $R>0$. Given $s_0\in\mathbb C\setminus\{-\mu_D\lambda_j-\mu_a\}_{j=1}^\infty$, we can find $R>0$ such that $s_0\in \mathbb C_R\setminus \{-\mu_D\lambda_j-\mu_a\}_{j=1}^\infty$, which means $\sum_{n=1}^\infty c_{z,n}p_{k,n}(\mu_D\lambda_n+\mu_a)(s+\mu_D\lambda_n+\mu_a)^{-1}$ is analytic on $\mathbb C\setminus \{-\mu_D\lambda_j-\mu_a\}_{j=1}^\infty$. 

Now, let us show the analyticity of  
 $$\sum_{k=1}^{K}(e^{-t_{k-1}s}-e^{-t_ks})\Big[\sum_{n=1}^\infty c_{z,n}p_{k,n}(\mu_D\lambda_n+\mu_a)(s+\mu_D\lambda_n+\mu_a)^{-1}\Big]$$ on $\mathbb C^+$. 
 For the case of $K$ is finite, from the result that $\sum_{n=1}^\infty c_{z,n}p_{k,n}(\mu_D\lambda_n+\mu_a) (s+\mu_D\lambda_n+\mu_a)^{-1}$ is holomorphic on $\mathbb C^+$, we can deduce the desired result straightforwardly. 
 If $K$ is infinity, for $s\in C^+$ we see that 
 $|(\mu_D\lambda_n+\mu_a)(s+\mu_D\lambda_n+\mu_a)^{-1}|\le 1$ and $|e^{-t_ks}|\le 1$, so that  
\begin{equation*}
\sum_{k=1}^{K}|e^{-t_{k-1}s}-e^{-t_ks}|\Big[\sum_{n=1}^\infty 
 |c_{z,n}p_{k,n}|\ |(\mu_D\lambda_n+\mu_a)(s+\mu_D\lambda_n+\mu_a)^{-1}|\Big]
 \le 2 \sum_{k=1}^K \sum_{n=1}^\infty |c_{z,n}p_{k,n}|.
\end{equation*}
Then the above proof and Assumption \ref{condition_p_regularity} give the uniform convergence of the above series on $\mathbb C^+$. Sequentially, we can deduce the desired result and the proof is complete. 
\end{proof}

\subsection{Proof of Theorem \ref{uniqueness}.}
Here we will prove the main theorem. To shorten our proof, we define 
\begin{align*}
 P_{z,k}(s)&:=\sum_{n=1}^\infty c_{z,n}p_{k,n}(\mu_D\lambda_n+\mu_a)(s+\mu_D\lambda_n+\mu_a)^{-1},\\
 \tilde P_{z,k}(s)&:= \sum_{n=1}^\infty 
 c_{z,n}\tilde p_{k,n}(\tilde\mu_D\lambda_n+\tilde\mu_a) (s+\tilde\mu_D\lambda_n+\tilde\mu_a)^{-1}.
\end{align*}

\begin{proof}[Proof of Theorem \ref{uniqueness} (\romannumeral 1) ] 
Given $\mu_a=\tilde \mu_a$ and $\mu_D=\tilde \mu_D$, from \eqref{laplace} and Lemma \ref{analytic}, we have that for $s\in\mathbb C^+$ and $z\in\Gamma_d$, 
\begin{equation}\label{equality_2}
 \sum_{k=1}^{K}(e^{-t_{k-1}s}-e^{-t_ks})P_{z,k}(s)
 =\sum_{k=1}^ K(e^{-t_{k-1}s}-e^{-t_ks})
 \tilde P_{z,k}(s). 
\end{equation}
We first prove that $p_{1,n}=\tilde p_{1,n}$ for $n\in\mathbb N^+$. Multiplying $e^{(t_0+\epsilon)s}$ with sufficiently small $\epsilon>0$ such that 
$\epsilon<t_1-t_0$ on \eqref{equality_2} gives 
\begin{equation}\label{equality_4}
\begin{aligned}
e^{\epsilon s} [P_{z,1}(s)-\tilde P_{z,1}(s)]
 =&e^{(\epsilon+t_0-t_1)s}P_{z,1}(s)  -\sum_{k=2}^{K}(e^{(\epsilon+t_0-t_{k-1})s}-e^{(\epsilon+t_0-t_k)s}) P_{z,k}(s)\\
 &-e^{(\epsilon+t_0-t_1)s}\tilde P_{z,1}(s)+\sum_{k=2}^K(e^{(\epsilon+t_0-t_{k-1})s}-e^{(\epsilon+t_0-t_k)s})\tilde P_{z,k}(s).
\end{aligned}
\end{equation}
For $s\in \mathbb C^+$, not hard to see that  
$$|(\mu_D\lambda_n+\mu_a)(s+\mu_D\lambda_n+\mu_a)^{-1}|\le (\mu_D\lambda_n+\mu_a) 
(\re s +\mu_D\lambda_n+\mu_a)^{-1}\le 1.$$
This with Assumption \ref{condition_p_regularity} yields that 
\begin{equation*}
\begin{aligned}
\lim_{\re s\to\infty}\Big|\sum_{k=2}^{K}(e^{(\epsilon+t_0-t_{k-1})s}&-e^{(\epsilon+t_0-t_k)s}) P_{z,k}(s) \Big|\\
 &\le \lim_{\re s\to\infty} 2e^{(\epsilon+t_0-t_1)\re s} 
 \sum_{k=2}^{K}\sum_{n=1}^\infty |c_{z,n}p_{k,n}| =0.
 \end{aligned} 
\end{equation*}  
Analogously, we can show that other terms in the right side of \eqref{equality_4} tend to zero as $\re s\to\infty$. Now we have $$
\lim_{\re s\to\infty} e^{\epsilon s}[P_{z,1}(s)-\tilde P_{z,1}(s)] = 0\ \text{for}\ z\in \Gamma_d. $$
Then with Lemma \ref{lemma_uniqueness} we have $p_{1,n}=\tilde p_{1,n} \ \text{for}\ n\in\mathbb N^+,$ namely $\|p_1-\tilde p_1\|_{L^2(\Omega)}=0$. 

Next we need to show $p_{k,n}=\tilde p_{k,n}$ for $n\in\mathbb N^+$ and $k>1$. From the result $p_{1,n}=\tilde p_{1,n}, \ n\in\mathbb N^+,$ we have $P_{z,1}(s)=\tilde P_{z,1}(s)$. Inserting it into \eqref{equality_2} yields that for $s\in \mathbb C^+$ and $z\in\Gamma_d$, 
\begin{equation*}
 \sum_{k=2}^{K}(e^{-t_{k-1}s}-e^{-t_ks})P_{z,k}(s)
 =\sum_{k=2}^{ K}(e^{-t_{k-1}s}-e^{-t_ks})\tilde P_{z,k}(s).
\end{equation*}
Following the above proof gives that $\|p_2-\tilde p_2\|_{L^2(\Omega)}=0$. Continuing this argument, we conclude that $\|p_k-\tilde p_k\|_{L^2(\Omega)}=0$ for $k=1,\cdots,K$. The proof is complete. 
\end{proof}

Before to show Theorem \ref{uniqueness} (\romannumeral 2), we state \cite[Lemma 3.5]{RundellZhang:2020} below. 
\begin{lemma}\label{lemma_uniqueness_2}
 Let $\{\tau_n\}_{n=1}^\infty$ be an absolutely convergent complex sequence and $\{\gamma_n\}_{n=1}^\infty$ be a real sequence satisfying  $0\le \gamma_1<\gamma_2<\cdots$ and $\gamma_n\to \infty.$
For the complex series $\sum_{n=1}^\infty \tau_n e^{-\gamma_n t}$ defined on $\mathbb{C}^+$, if the set of its zeros on $\mathbb{C}^+$ has an accumulation point, then $\{\tau_n\}_{n=1}^\infty=\{0\}$.
\end{lemma}

\begin{proof}[Proof of Theorem \ref{uniqueness} (\romannumeral 2)] 
 Given the condition \eqref{condition_mu_a} and $\mu_D=\tilde \mu_D$, firstly let us show $\mu_a=\tilde\mu_a$. 
 Assume not, without loss of generality, we can set $\mu_a<\tilde\mu_a$.  From Lemma \ref{data_formula}, it holds that for $t\in (t_0,t_1],$  
 \begin{align*}
-\frac{\partial u}{\partial\nu}(z,t)&=\sum_{n=1}^\infty c_{z,n}p_{1,n}(1-e^{-(\mu_D\lambda_n+\mu_a)(t-t_0)}),\\
-\frac{\partial \tilde u}{\partial\nu}(z,t)&=\sum_{n=1}^\infty c_{z,n}\tilde p_{1,n}(1-e^{-(\mu_D\lambda_n+\tilde \mu_a)(t-t_0)}).
 \end{align*}
Then we have 
\begin{equation}\label{equality_5}
\sum_{n=1}^\infty c_{z,n}p_{1,n}(1-e^{-(\mu_D\lambda_n+\mu_a)(t-t_0)})=\sum_{n=1}^\infty c_{z,n}\tilde p_{1,n}(1-e^{-(\mu_D\lambda_n+\tilde \mu_a)(t-t_0)}) 
\end{equation}
for $t\in (t_0,t_1].$ The condition $\mu_a<\tilde \mu_a$ gives that 
$\mu_D\lambda_1+\mu_a<\mu_D\lambda_j+\mu_a$ for $j>1$ and $\mu_D\lambda_1+\mu_a<\mu_D\lambda_j+\tilde\mu_a$ for $j\ge 1$. Using Lemma \ref{lemma_uniqueness_2}, we have that $\sum_{\lambda_n=\lambda_1} c_{z,n}p_{1,n}=0$ for $z\in \Gamma_d$. This together with Lemma \ref{lemma_uniqueness_1} leads to $p_{1,n}=0$ for $\lambda_n=\lambda_1$, which contradicts with condition \eqref{condition_mu_a}. 
So we have $\mu_a=\tilde \mu_a$. 

With $\mu_a=\tilde \mu_a$ and the proof for Theorem \ref{uniqueness} (\romannumeral 1), we deduce that $\|p_k-\tilde p_k\|_{L^2(\Omega)}=0$ for $k=1,\cdots,K$. 
The proof is complete. 
\end{proof}

\begin{proof}[Proof of Theorem \ref{uniqueness} (\romannumeral 3)] 
 With condition \eqref{condition_mu_a} and $\mu_a=\tilde 
\mu_a$, let us show $\mu_D=\tilde\mu_D$.  
 Assuming that $\mu_D<\tilde\mu_D$, from \eqref{equality_5} we have 
 \begin{equation*}
\sum_{n=1}^\infty c_{z,n}p_{1,n}(1-e^{-(\mu_D\lambda_n+\mu_a)(t-t_0)})
=\sum_{n=1}^\infty c_{z,n}\tilde p_{1,n}(1-e^{-(\tilde\mu_D\lambda_n+ \mu_a)(t-t_0)}).
\end{equation*} 
 With Lemmas \ref{lemma_uniqueness_1} and  \ref{lemma_uniqueness_2}, following the proof of Theorem \ref{uniqueness} (\romannumeral 2) we can get $p_{1,n}=0$ if $\lambda_n=\lambda_1$, which is a contradiction. Hence, $\mu_D=\tilde \mu_D$. 
 
 With the result $\mu_D=\tilde \mu_D$ and the proof of Theorem \ref{uniqueness} (\romannumeral 1), we have $\|p_k-\tilde p_k\|_{L^2(\Omega)}=0$ for $k=1,\cdots,K$. The proof is complete. 
\end{proof}

\section{Numerical inversions.}\label{sec_num}
Noting that $\mu_D:=1/(3\mu_s(1-g))$, where $g$ is the anisotropy parameter and known, we may consider the numerical inversion of $\textbf{a}:=(\mu_a,\mu_s,\mu_f)$. In practice, for the FDOT people inject the light from a laser source to the biological tissue via one of the source fibers at boundary, then measure the amount of transmitted light at all the boundary detector locations using the detector fibers. This process is repeated for all the source locations. Now we show the mathematical description of our FDOT as follows.

Let $\Gamma_s$ and $\Gamma_d$ be the finite set of source locations and detector locations, respectively. Let $\mathcal{T}\subset(0,\infty)$ be a time interval in which we take the measurements. Then, we can denote  the exact boundary measurements corresponding to any given input $\textbf{a}\in\mathcal A$ with $\big(x_s^{(n)},x_d^{(k)}\big)\in\Gamma_s\times\Gamma_d$ and $t_i\in\mathcal{T}$ by $u_m[\textbf{a}]\big(x_d^{(k)},t_i;x_s^{(n)}\big)$. Here  $\mathcal A$ is the admissible set. Taking the measured data at $N$ different excitation sources, $K$ different detectors and $I$ different times, the computational inverse problem is to determine the unknown $\textbf{a}$ from the following set of measurements
\begin{equation}\label{mea}
\left\{u_m\big(x_d^{(k)},t_i;x_s^{(n)}\big): n=1,\cdots,N,\  k=1,\cdots, K,\ i=1,\cdots,I\right\},
\end{equation}
where the set $\{x_d^{(k)}\}_{k=1}^K$ and the set $\{t_i\}_{i=1}^I$ may be varied for different source locations.
More precisely, for each given source $x_s^{(n)}$, the process of obtaining the finite set of measurements (\ref{mea}) can be described as
\begin{equation}\label{sun4-3}
\mathbb{K}_n:\mathcal A\to \mathbb{R}^{K\times I}, \   \mathbb{K}_n(\textbf{a})= u_m[\textbf{a}]\big(x_d^{(k)}, t_i;x_s^{(n)}\big),\  k=1,\cdots,K,\ i=1,\cdots,I.
\end{equation}
For convenience, we set $M=N\times K$ and denote $\{\omega_m:=\big(x_s^{(m)},x_d^{(m)}\big)\}_{m=1}^M$ be the $M$ different S-D pairs we used in \eqref{sun4-3}, where the source location or detector location for different S-D pairs may be same. Then we can rewrite (\ref{sun4-3}) as 
\begin{equation}\label{sun4-4}
\mathbb{K}: \mathcal A\to \mathbb{R}^{M\times I}, \quad \mathbb{K}(\textbf{a})=h,
\end{equation}
where $h:=u_m(\hat\omega,\hat t)$ is the measured data corresponding to S-D pairs $\hat\omega:=(\omega_1,\cdots,\omega_M)$ and time points $\hat t\in \mathcal{T}^I$. Then our FDOT is to solve the equation \eqref{sun4-4}. 

In the follows, unless in particular cases, we always take the physical parameters as
\begin{eqnarray*}
c=0.219\;\rm {mm}/{ps}, \quad \mu_{s}'=1.0 \; \rm{mm}^{-1}, \quad \mu_a=0.1\;\rm{mm}^{-1},\quad \beta=0.01\;\rm{mm}^{-1}.
\end{eqnarray*}
We use mm as the unit of length. For simplicity, we consider the zero-lifetime case of \eqref{emission-um}, i.e. the lifetime in the source term $S_0[\mu_f,u_e]$ is $\tau=0$. Further, we are here focusing on thick ($>1$ cm) or large volume tissue ($>10$ cm$^3$) like chest and thus we may set
\begin{equation*}
\Omega:=\mathbb{R}^3_+=\left\{(x_1,x_2,x_3): (x_1,x_2)\in\mathbb{R}^2, x_3>0\right\}
\end{equation*}
with the boundary $\partial\Omega:=\left\{(x_1,x_2,0): (x_1,x_2)\in\mathbb{R}^2 \right\}$. 

\subsection{Identify a stationary target by peak detection.}

Supposing $\mu_s$ and $\mu_a$ are known, we consider to identify the stationary $\mu_f(x)$. For the ideal case that the size of fluorescence target is very small, we may concern the point target, i.e. 
\begin{equation*}\label{delta}
\mu_f(x)=P\delta(x-x_c),
\end{equation*}
where $x_c=(x_{c1},x_{c2},x_{c3})\in\Omega$ is the location of point target and $P>0$ is the concentration of fluorescent target. Then, our FDOT is to identify $\textbf{a}=(x_{c1},x_{c2},x_{c3},P)$. For each given $x_d\in\partial\Omega$ and $x_s\in\partial\Omega$, we have that
\begin{equation}\label{measure}
u_m(t) = C(t)P \int_0^t \frac{1}{\big[(t-s)s\big]^{3/2}}  e^{-\frac{|x_d-x_c|^2}{4c\mu_D(t-s)}} e^{-\frac{|x_s-x_c|^2}{4c\mu_Ds}} K_3(0,x_{c3};t-s) K_3(x_{c3},0;s) \,ds,
\end{equation}
where $C(t):=\frac{ e^{-c\mu_a t}}{16\pi^3 c{\mu_D}^2}$ and
\begin{equation*}\label{K3}
K_3(x_3,y_3;t) :=
1 - \beta\sqrt{\pi c\mu_Dt} \,\exp{\left(\left(\frac{x_3+y_3+2\beta c\mu_Dt}{\sqrt{4c\mu_Dt}}\right)^2\right)}
 \mathop{\mathrm{erfc}}
\left(\frac{x_3+y_3+2\beta c\mu_Dt}{\sqrt{4c\mu_Dt}}\right).
\end{equation*}
The expression \eqref{measure} is a temporal point-spread function (TPSF). Before introducing the inversion strategy by peak detection, we define two geometric concepts for TPFS as follows (see Figure \ref{peaktime-fig}): 
\begin{itemize}
 \item  Peak intensity $u_m^{\rm peak}$:  the TPSF maximum;
 \item Peak time $t_{\rm peak}$: the temporal position of TPSF maximum.
\end{itemize}

The existence and uniqueness of the peak time for \eqref{measure} are mathematically investigated in \cite{Sun21}. There holds that
\begin{equation}\label{peak_int}
u_m^{\rm peak} := u_m^{\rm peak}[x_s,x_d] \approx C(t_{\rm peak}) P\left(\sqrt{\frac{\pi}{A}} + \sqrt{\frac{\pi}{B}} \right)
\exp\left(-\frac{2A+2B}{t_{\rm peak}} \right)  {t_{\rm peak}}^{-\frac{3}{2}}
\end{equation}
and
\begin{equation}\label{peak_time}
t_{\rm peak} :=t_{\rm peak}[x_s,x_d] \approx \frac{1}{4c\mu_a} \left( -3+\sqrt{9+32c\mu_a (A+B)}\right),
\end{equation}
where 
\begin{equation*}
A:=\frac{|x_d-x_c|^2}{4c\mu_D}, \quad B:=\frac{|x_s-x_c|^2}{4c\mu_D}.
\end{equation*}

\begin{figure}[h!]
\centering
\includegraphics[width=0.55\textwidth,height=0.25\textheight]{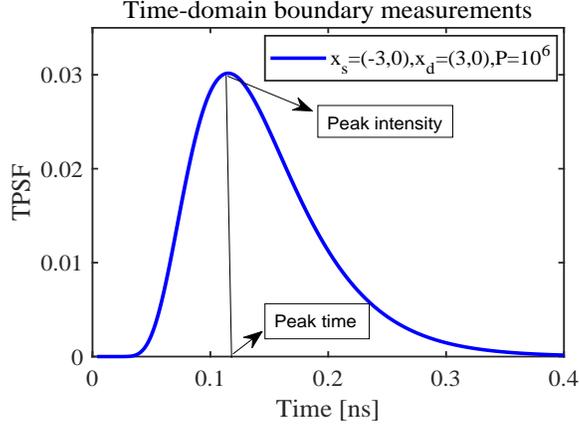}\\
\caption {The profile of TPSF for a point target located at $x_c=(0,0,5)$.}\label{peaktime-fig}
\end{figure}

From above expressions we can see that the shape of the profile of $u_m$ is actually dominated by the distances $|x_d-x_c|$ and $|x_s-x_c|$. Fixing the distance between detector and source, we have stronger peak intensity as the S-D pair closing to the target and we can obtain the strongest one when the S-D pair satisfies $\frac{x_{s1}+x_{d1}}{2}=x_{c1}$ and $\frac{x_{s2}+x_{d2}}{2}=x_{c2}$. This implies that we can identify $(x_{c1},x_{c2})$ by comparing the peak intensity corresponding to different S-D pairs. Since the peak time is not dependent on the concentration $P$, we can further identify the depth of target ($x_{c3}$) from the expression of peak time as in \eqref{peak_time}. Finally, the concentration $P$ can be identified by \eqref{peak_int}. Hence the horizontal location, vertical location and concentration $P$ of the target can be separately determined by scanning the S-D pair at the boundary surface. This inversion strategy by peak detection is summarized as follows.

\begin{itemize}
\item  {Step 1} (identify $(x_{c1},x_{c2})$): Fix the distance $|x_s-x_d|$ and search an S-D pair $(x_s^*,x_d^*)$ such that 
\begin{equation*}
{u_{m,*}^{\rm peak}}:=u_m^{\rm peak}[x_s^*,x_d^*] = \max_{(x_s,x_d)\in\partial\Omega\times\partial\Omega} \big\{ u_m^{\rm peak}[x_s,x_d]\big\}.
\end{equation*}
Then the horizontal location of the target is identified by
\begin{equation*}\label{hori}
(x_{c1},x_{c2})=\left(\frac{x_{s1}^*+x_{d1}^*}{2}, \frac{x_{s2}^*+x_{d2}^*}{2}\right).
\end{equation*}
\item {Step 2} (identify $x_{c3}$): 
By \eqref{peak_time}, the depth of the target is identified by 
\begin{equation*}\label{verti}
x_{c3}=\sqrt{c^2\mu_D\mu_a \big(t_{\rm peak}^* \big)^2+\frac{3}{2}c\mu_Dt_{\rm peak}^*-\frac{1}{4}|x_s^*-x_d^*|^2},
\end{equation*}
where $t_{\rm peak}^*$ is the measured peak time for the S-D pair $(x_s^*,x_d^*)$ obtained from Step 1.
\item {Step 3} (identify $P$): 
By \eqref{peak_int}, the concentration of the target is identified by 
\begin{equation*}\label{P}
P=\left. {u_{m,*}^{\rm peak}} \middle/ \left\{C(t_{\rm peak}^*) \left(\sqrt{\frac{\pi}{A^*}} + \sqrt{\frac{\pi}{B^*}} \right)\exp\left(-\frac{2A^*+2B^*}{t_{\rm peak}^*} \right)  {t_{\rm peak}^*}^{-\frac{3}{2}}\right\} \right..
\end{equation*}
\end{itemize}

Now we test the accuracy of \eqref{peak_int} and \eqref{peak_time} to ensure the effectiveness of above inversion strategy. We start with the point target given in the following example. 

\medskip

{\bf Example 1.} Set $x_s=(-3,0,0), x_d=(3,0,0)$ and $P=10^6 \; {\rm mm}^{-1}$. Suppose the point target is located at
\begin{equation*}
x_c=(0,0,x_{c3}).
\end{equation*}
We take $x_{c3}\in [0.05,6]$ and show the comparisons of \eqref{peak_int} in Figure \ref{peakpoint-fig} (a)  and \eqref{peak_time} in Figure \ref{peakpoint-fig} (b), respectively.

\begin{figure}[h!]
\centering
\includegraphics[width=1\textwidth,height=0.25\textheight]{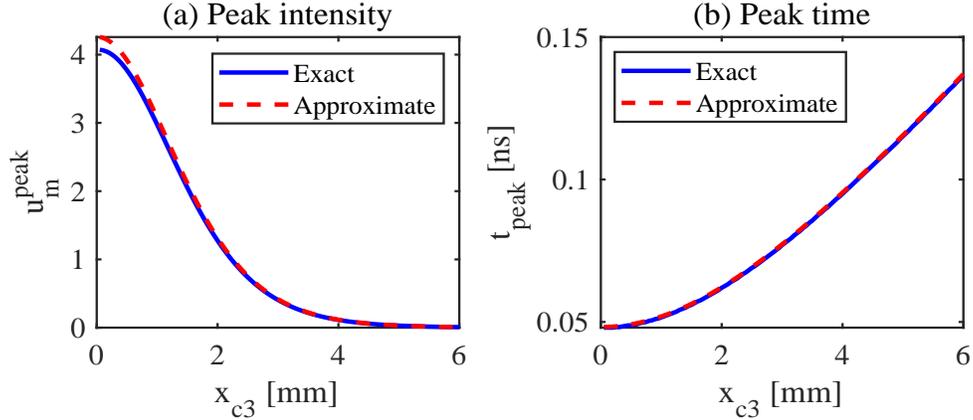}\\
\caption {Comparisons of \eqref{peak_int} and \eqref{peak_time} with the exact ones. (point target)}\label{peakpoint-fig}
\end{figure}

From Figure \ref{peakpoint-fig}, it can be observed that \eqref{peak_int} and \eqref{peak_time} provide a good approximation to the exact peak intensity and peak time, respectively. The approximate errors decrease as the depth of point target becomes deeper. 

By \eqref{peak_int} and \eqref{peak_time}, we can further give the expressions of approximate peak intensity and peak time for a small target. Suppose $\Omega_0$ is the distribution domain of this small target and $x_c$ is its center point.  By the middle value theorem we have $u_m[\Omega_0](t)\approx |\Omega_0|u_m[x_c](t)$, where $|\Omega_0|$ is the volume of small target. Then the peak time of small target will be given by \eqref{peak_time}, while its peak intensity can be approximated by $u_m^{\rm peak}[\Omega_0]\approx |\Omega_0|u_m^{\rm peak}[x_c]$, where $u_m^{\rm peak}[x_c]$ is the peak intensity of the point target located at $x_c$. Now we  have tested the accuracy of \eqref{peak_int} and \eqref{peak_time} to approximate the peak intensity and peak time of a small target. The case of a cubic target will be  given in the following example. 

\medskip

{\bf Example 2.} Set $x_s=(-3,0,0)$, $x_d=(3,0,0)$ and $P=10^6 \; {\rm mm}^{-1}$. Suppose the cubic target with side length $L$ is located at $x_c=(0,0,5)$. We take $L\in [0.01,1]$ and show the comparisons of the above approximate expressions with the exact ones in Figure \ref{peak-fig-cubic} (a) and (b), respectively.

\begin{figure}[h!]
\centering
\includegraphics[width=1\textwidth,height=0.23\textheight]{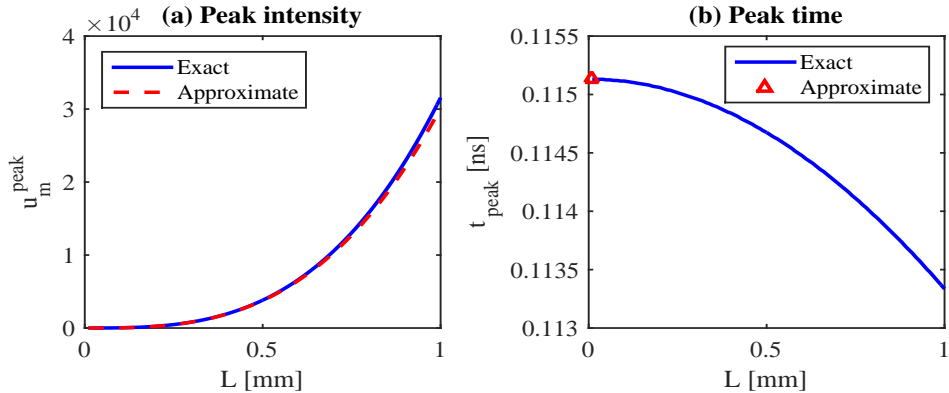}\\
\caption {Comparisons of the approximate peak intensity and peak time with the exact ones. (cubic target)}\label{peak-fig-cubic}
\end{figure}

Figure \ref{peak-fig-cubic} shows that the peak intensity and peak time of a small target can be approximated well by the ones of a point target located at its center. Even if $L=1$ mm, the relative error for peak time is $1.59e-2$, implying that  \eqref{peak_time} is still useful for the case of small target.

Summarizing above numerical results, we realize that the explicit expressions of peak intensity and peak time are  accurate, and hence our proposed inversion strategy is effective to identify the location of a small target. However, we point out that the validity of this strategy depends on the relative distance $|A-B|$, which requires $|A-B|$ is small enough \cite{Sun21}.

\subsection{Identify a target by iterative algorithm. }

The FDOT is to solve the operator equation \eqref{sun4-4}. For exact observation data, if one of $\mu_s$ and $\mu_a$ is known, this equation has a unique solution from Theorem \ref{uniqueness}. For given noisy data $h^\delta$ satisfying
\begin{equation*}\label{noisy-data}
\|h^\delta-h\|_2 \le \delta,
\end{equation*}
we consider the approximate solution of \eqref{sun4-4} by the minimizer of the cost functional
\begin{equation}\label{cost-J}
\mathcal{J}_{\delta,\alpha}(\mathbf{a}):=\|\mathbb{K}(\mathbf{a})-h^\delta\|^2_2+
\alpha\|\mathbf{a}\|^2_2,
\end{equation}
where $\alpha>0$ is the regularization parameter. Supposing the fluorescence target is often of regular shape with a constant absorbing coefficient, we can describe the distribution of $\mu_f$ in $\mathbb{R}^3_+$ as well as its interface by a finite dimensional vector $\textbf{a}\in \mathbb{R}^S$. Here $S$ is the dimension number of vector $\textbf{a}$. Then we can transform our FDOT into a finite dimensional inverse problem. For instance, if we suppose the stationary target is a cuboid, we have
\begin{eqnarray*}\label{cuboid}
\mu_f(x)=
\begin{cases}
P, &x\in \Omega_0:=\left\{x:\,x_1\in(a_1,b_1),\,x_2\in(a_2,b_2),\,x_3\in(a_3,b_3)\right\},\\
0, &x\not\in \Omega_0,
\end{cases}
\end{eqnarray*}
where $P>0$ is the concentration of fluorescent target and $b_3>a_3>0$. 

Now we give a sketch for solving this minimization problem. In fact, solving \eqref{cost-J} can be transformed to solve a normal equation combining with an iteration process:
\begin{equation*}\label{inv9}
\begin{cases}
\begin{aligned}
(\alpha I+\mathbb{G}^T\mathbb{G})\delta{\mathbf{a}_{j} }&=\mathbb{G}^T (h-\mathbb{K}(\mathbf{a}_j)),\\
{\mathbf{a}_{j} }+\delta{\mathbf{a}_{j} } &\to {\mathbf{a}_{j+1} }, \quad j=0,1,\cdots,
\end{aligned}
\end{cases}
\end{equation*}
where $\delta{\mathbf{a}_{j} }$ is a perturbation vector for any given ${\mathbf{a}_{j} }\in \mathcal{A}$, $j$ denotes the iterative number and ${\mathbf{a}_{0} }$ is the initial guess. Here $\alpha>0$ is selected by the discrepancy principle. The sensitivity matrix $\mathbb{G}$ is given by
\begin{eqnarray*}\label{inv7}
\mathbb{G} ={(g_{qs})}_{Q\times S},\quad Q=N\times K\times I,
\end{eqnarray*}
where $g_{qs}$ is the derivative of $u_m$ on the $s$-th component of $\textbf{a}$, corresponding to the $q$-th S-D pair of total $N\times K$ S-D pairs. We terminate the
iteration process by $\|\delta\mathbf{a}_j\|_2\leq \eta$ for some  specified $\eta>0$, which is taken as $\eta=10^{-8}$ for our numerics.

\begin{figure}[h!]
\centering
\includegraphics[width=0.5\textwidth,height=0.3\textheight]{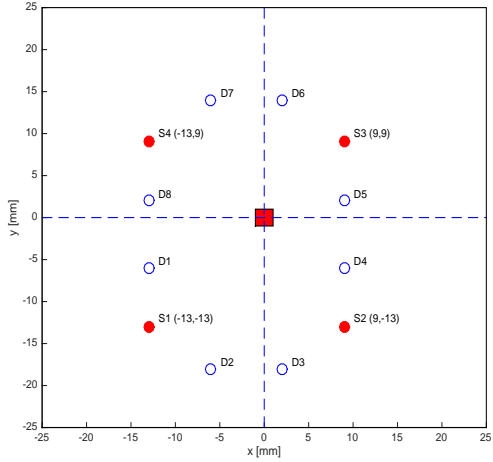}\\
\caption {Horizontal projection of the configuration: the middle red cubic is the projection of a target to the boundary $\partial\Omega$, the red points are the excitation points, and the blue circles are the detector points.}\label{setup}
\end{figure}

\medskip

{\bf Example 3.} Set $\mu_s=10\;\rm{mm}^{-1}$ and $\mu_a=2 \; \rm{mm}^{-1}$. We suppose the distribution of cuboid target is
\begin{eqnarray*}\label{inv11}
\Omega_0:=\left\{ (x_1,x_2,x_3): x_1\in(-1,1),\  x_2\in(-1,1),\  x_3\in(9,11) \right\}
\end{eqnarray*}
with the concentration $P=0.5\; \rm{mm}^{-1}$, i.e.
\begin{equation*}
\mu_f(x)=
\begin{cases}
P, &x\in \Omega_0,\\
0, &x\not\in \Omega_0.
\end{cases}
\end{equation*}
For this example,  we observe the measurements from the setup shown in Figure \ref{setup} and test our inversion algorithm for the following two cases:
\begin{itemize}\label{example_3_two_cases}
\item[(3a)] given $(\mu_s,\mu_a)$, identify stationary $\mu_f(x)$;
\item[(3b)] simultaneously identify $(\mu_s,\mu_a,\mu_f(x))$. 
\end{itemize}

As in Figure \ref{setup}, we have 4 excitation sources and observe the measurements at 8 detector points for each source point.  Hence we have 32 S-D pairs in total for this setup. The sources and detectors are located at
\begin{equation*}
S_1=(-13,-13),\;S_2=(9,-13),\;S_3=(9,9),\;S_4=(-13,9)
\end{equation*} 
and
\begin{equation*}
D_1=(-13,-6),\;D_2=(-6,-18),\;D_3=(2,-18),\;D_4=(9,-6),
\end{equation*}
\begin{equation*}
D_5=(9,2),\;D_6=(2,14),\;D_7=(-6,14),\;D_8=(-13,2).
\end{equation*}

For each S-D pair, we select the peak time $t_{\rm peak}$ and choose 20 time points $\hat t:=[t_{\rm peak}-10\Delta t: \Delta t: t_{\rm peak}+9\Delta t]$ with time step $\Delta t=2$ ps, so that the measurement $h$ is a $640$-dimensional vector. The noisy data $h^\delta$ of $h$ is described by
\begin{equation}\label{noisydata}
h^\delta = h(1+\zeta\epsilon),
\end{equation}
where $\epsilon>0$ is the noise level and $\zeta$ is a random standard Gaussian noise.

For the case $(3a)$, we set the exact solution as  
\begin{eqnarray*}\label{case1}
\textbf{a}_{\rm exa}=(-1,1,-1,1,9,11,0.5).
\end{eqnarray*}
Noting the randomness of the noisy data \eqref{noisydata}, we carry out 10-time recoveries from different noisy data set for each given noise level and take an average of the recoveries. The recovered results from the initial guess $\mathbf{a}_0=(-5.1, -4.9, -2.1, -1.9, 5.9, 6.1, 0.1)$ are listed in Table \ref{tab1}, where $\textbf{a}_{\rm rec}^{\rm avr}$ denotes the average of 10-time recoveries and $Err$ is the $L^2$ relative error in recoveries given by
\begin{equation*}\label{Err}
Err:=\frac{\norm{\textbf{a}_{\rm exa}-\textbf{a}_{\rm rec}^{\rm avr}}_2}{\norm{\textbf{a}_{\rm exa}}_2}.
\end{equation*}
The results of the 10-time recoveries are plotted in Figure \ref{case1-10time}. 

\begin{table}[h!]
\centering
\caption{The average of 10-time recoveries for the case $(3a)$.}\label{tab1}
\begin{tabular}{|c|c|c|c|c|}
\hline
$\epsilon$  &${{\textbf{a}}_{\rm rec}^{\rm avr}}=(a_1^*, b_1^*, a_2^*, b_2^*, a_3^*, b_3^*, P^*)$   &$Err$\\
\hline
$5\%$   & $(-0.85, 0.92, -0.85, 0.92, 9.01, 10.78, 0.81)$    & $3.93e-2$\\
$1\%$   & $(-0.98, 0.99, -0.98, 0.99, 9.00, 10.97, 0.52)$    & $3.50e-3$\\
\hline
\end{tabular}
\end{table}

\begin{figure}[h!]
\centering
\includegraphics[width=1\textwidth,height=0.25\textheight]{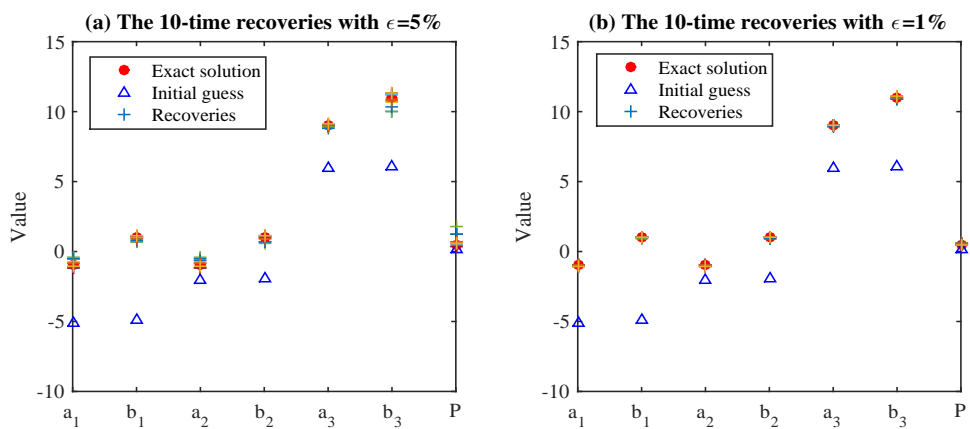}\\
\caption {The 10-time recoveries with different $\epsilon$ for the case $(3a)$.}\label{case1-10time}
\end{figure}

For the case $(3b)$, we need to consider the simultaneous inversion of $(\mu_s,\mu_a,\mu_f(x))$. This implies that we have 9 unknown parameters and the exact solution is 
\begin{eqnarray*}\label{case2}
\textbf{a}_{\rm exa}=(10,2,-1,1,-1,1,9,11,0.5).
\end{eqnarray*}
Setting the initial guess $\mathbf{a}_0=(5, 0.5, -5.1, -4.9, -2.1, -1.9, 5.9, 6.1, 0.1)$, the recovered results are shown in Table \ref{tab2} and Figure \ref{case2-10time}.

\begin{table}[h!]
\centering
\caption{The average of 10-time recoveries for the case $(3b)$.}\label{tab2}
\begin{tabular}{|c|c|c|c|c|}
\hline
$\epsilon$  &${{\textbf{a}}_{\rm rec}^{\rm avr}}=(\mu_s^*,\mu_a^*,a_1^*, b_1^*, a_2^*, b_2^*, a_3^*, b_3^*, P^*)$   &$Err$\\
\hline
$5\%$   & $(10.04,2.02,-1.01, 0.98, -0.99, 0.99, 9.04, 11.03, 0.80)$    & $2.12e-2$\\
$1\%$   & $(10.01,2.00, -1.00, 1.00, -1.00, 1.00, 9.01, 11.01, 0.53)$    & $2.20e-3$\\
\hline
\end{tabular}
\end{table}

\begin{figure}[h!]
\centering
\includegraphics[width=1\textwidth,height=0.25\textheight]{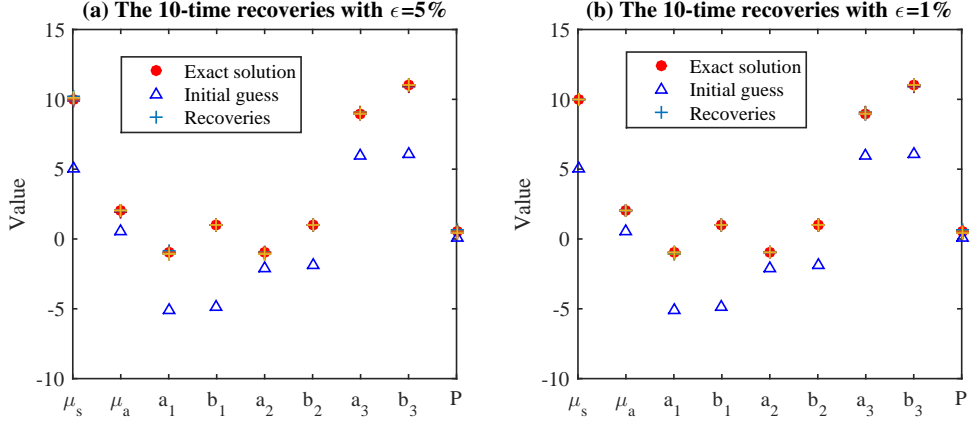}\\
\caption {The 10-time recoveries with different $\epsilon$ for the case $(3b)$.}\label{case2-10time}
\end{figure}

From Figures \ref{case1-10time} and \ref{case2-10time}, we see that the recoveries are satisfactory in view of the high ill-posedness of our inverse problem, which is caused by 
the nonlinearity of this inverse problem and the sparsity of the boundary measurements. Further, the results for case $(3b)$ indicate that we may reconstruct $(\mu_s,\mu_a,\mu_f(x))$ simultaneously, which is not covered by Theorem \ref{uniqueness}.  

In what follows, we further test the recovery performances of our inversion strategy for a time-dependent fluorescence target.

\medskip

{\bf Example 4.} Set $\mu_s=10\;\rm{mm}^{-1},\, \mu_a=2 \; \rm{mm}^{-1}$ and $P=0.5\; \rm{mm}^{-1}$. We recover a time-dependent cubic target for the following two cases: 
\begin{itemize}\label{example_4_two_cases}
\item[$(4a)$] fixed center $x_c=(0,0,10)$, time-dependent side length $L(t)=\frac{1}{2}+t$;

\item[$(4b)$] fixed side length $L=2$ mm, time-dependent center $x_c(t)$ with
\begin{equation*}\label{xct}
\begin{cases}
x_{c1}=-4+6t, \\
x_{c2}=-5+8t, \\
x_{c3}=8.
\end{cases}
\end{equation*}
\end{itemize}

In this example, we suppose either the side length or the center of the target depends on the time linearly and we are going to recover the time-dependent target by identifying its expansion coefficients in polynomial space. We also obtain the measurements from the setup in Figure \ref{setup}. For the case $(4a)$, setting the initial guess $(-2,-2,5,0.2)$ for $(x_c,P)$ and the initial guess $\frac{1}{10}+\frac{1}{5}t$ for $L(t)$, the recovered results are plotted in Figures \ref{t-cubic} and \ref{cubic-time}. For the case $(4b)$, setting the initial guess $(-1+2t,-2+3t,4)$ for $x_c(t)$ and the initial guess $(0.2,0.1)$ for $(L,P)$, the numerical performances of our inversion scheme are shown in Figures \ref{t-center} and \ref{cubic-time-center}.

Figures \ref{t-cubic}--\ref{cubic-time-center} show that the reconstructions are satisfactory. The recovered targets are in good agreement with the exact shape and its information of the time evolution process can be reconstructed well. Hence, summarizing the numerical results obtained from Examples 3--4, our proposed iterative algorithm is effective to recover the fluorescence target. Further, although we do the inversion only for a cubic or cuboid target in above examples, the position of the target with a general shape can be identified by cuboid approximation \cite{Sun19}.

\begin{figure}[h!]
\centering
\includegraphics[width=1\textwidth,height=0.27\textheight]{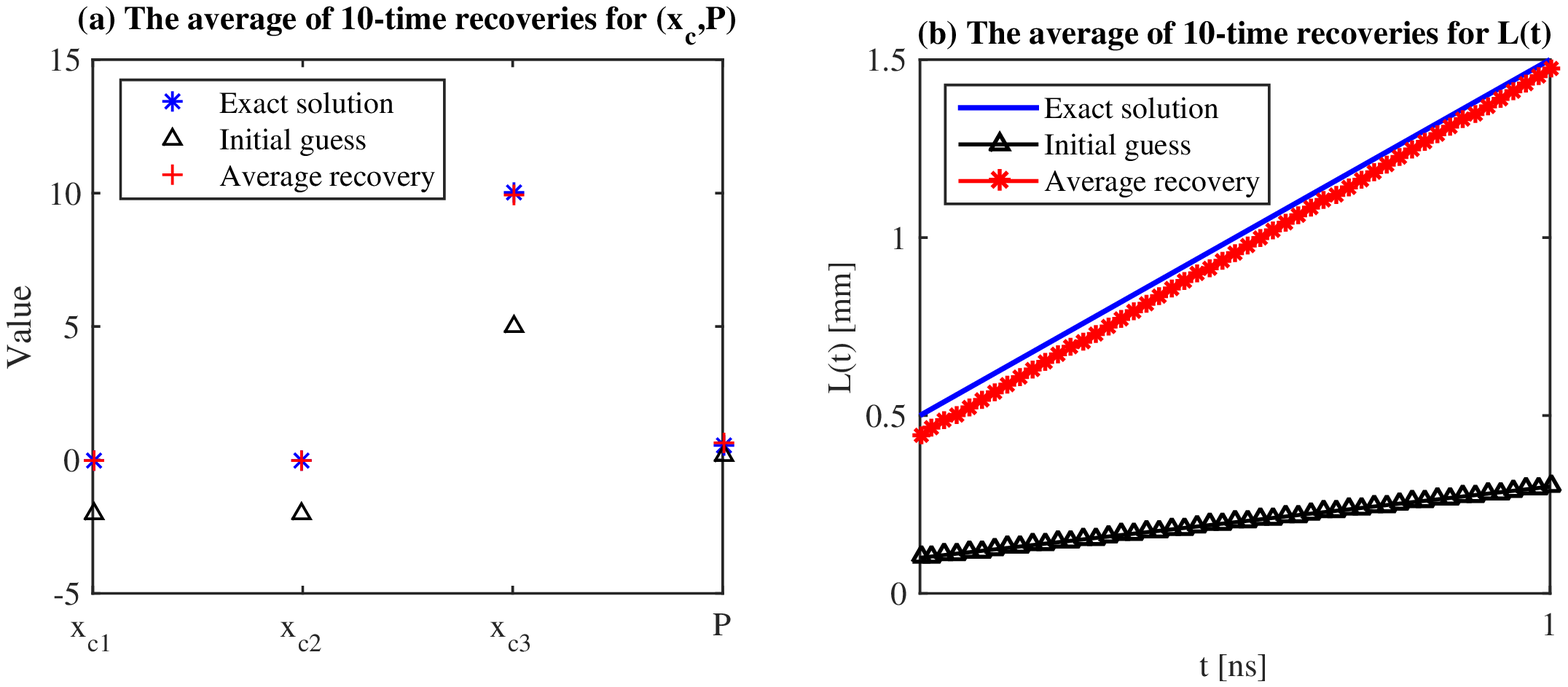}\\
\caption {The exact and recovered results for the case $(4a)$ ($\epsilon=1\%$).}\label{t-cubic}
\end{figure}

\begin{figure}[h!]
\centering
\includegraphics[width=1\textwidth,height=0.44\textheight]{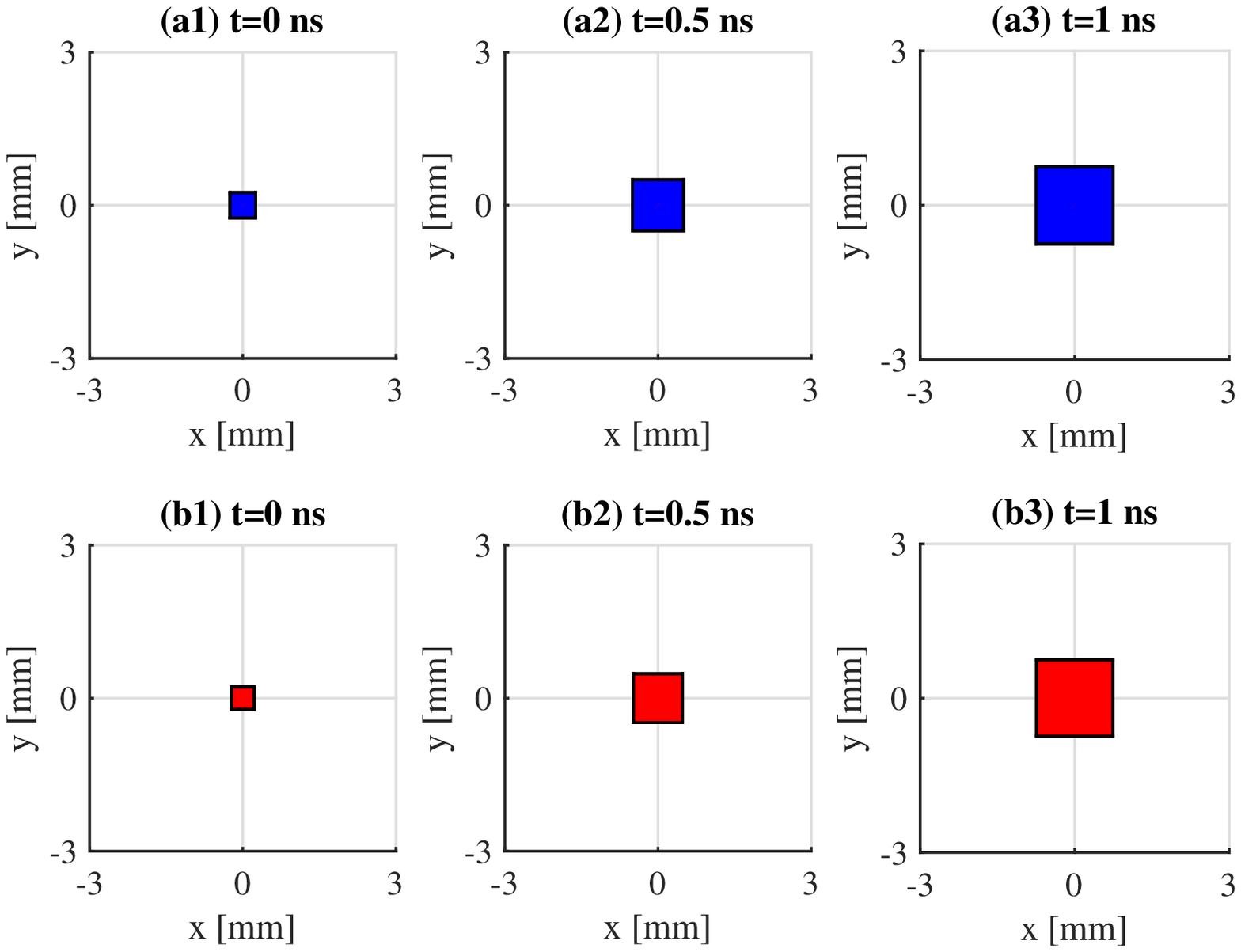}\\
\caption {The exact solution (blue) and the recovered solution (red) for the case $(4a)$ ($\epsilon=1\%$).}\label{cubic-time}
\end{figure}

\begin{figure}[h!]
\centering
\includegraphics[width=1\textwidth,height=0.18\textheight]{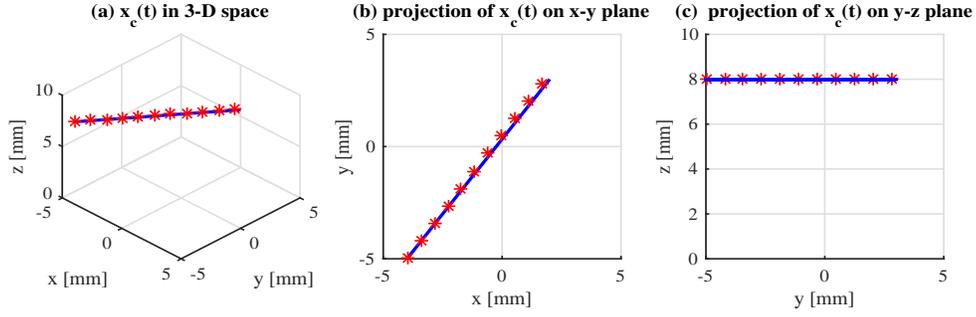}\\
\caption {The exact and recovered $x_c(t)$ for the case $(4b)$ ($\epsilon=1\%$).}\label{t-center}
\end{figure}

\begin{figure}[h!]
\centering
\includegraphics[width=1\textwidth,height=0.42\textheight]{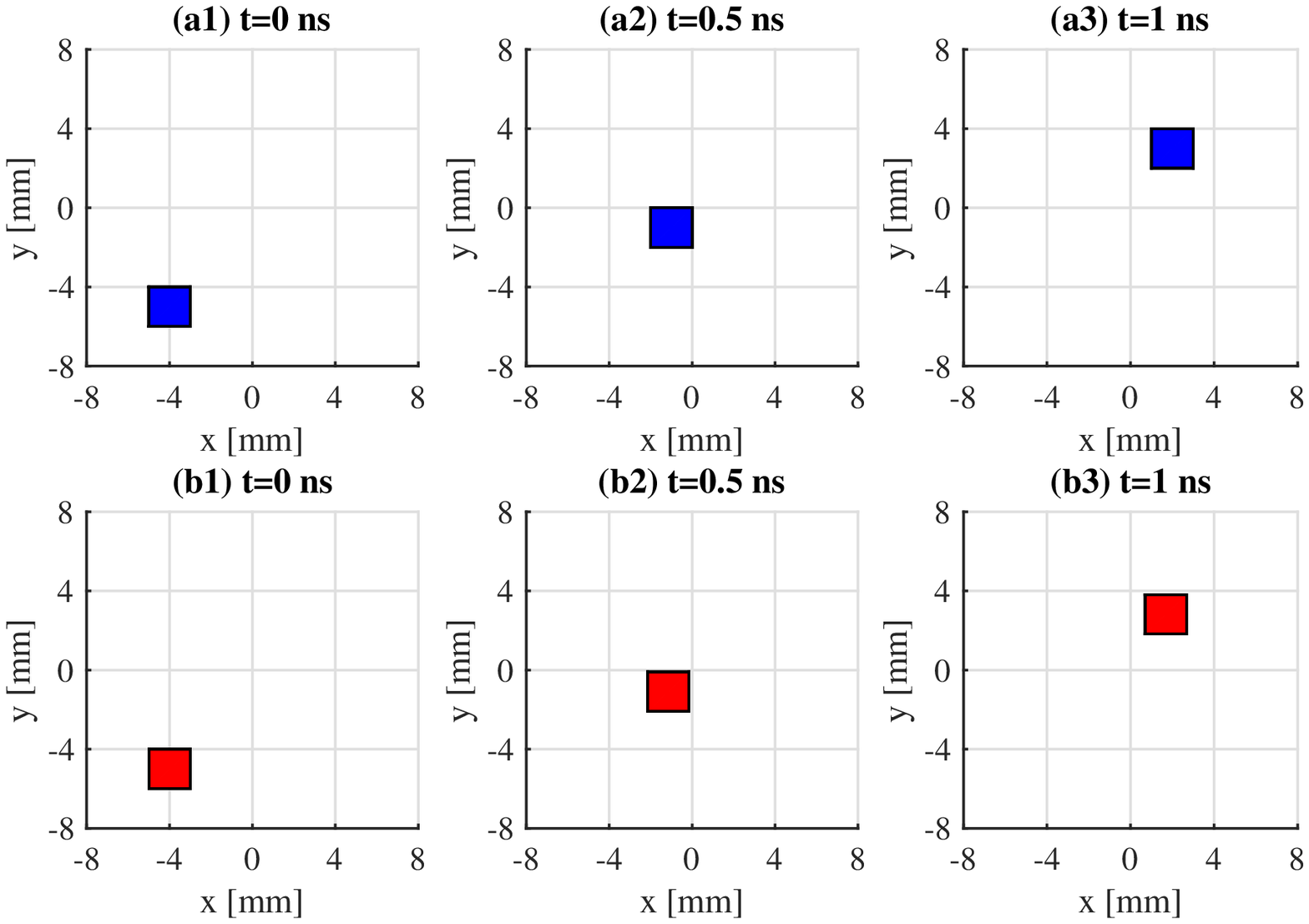}\\
\caption {The exact solution (blue) and recovered solution (red) for the case $(4b)$ ($\epsilon=1\%$).}\label{cubic-time-center}
\end{figure}

\section*{ Acknowledgement} 
Chunlong Sun is supported by National Natural Science Foundation of China (Grant No. 11971104) and by Natural Science Foundation of Jiangsu Province, China (Grant No. BK20210268). Zhidong Zhang is supported by National Natural Science Foundation of China (Grant No. 12101627).

\bibliographystyle{plainurl} 
\bibliography{ref}

\end{document}